\documentclass[letterpaper, 10 pt, conference]{ieeeconf}
\IEEEoverridecommandlockouts
\makeatletter

\def\ps@IEEEtitlepagestyle{%
    \def\@oddfoot{\mycopyrightnotice}%
    \def\@evenfoot{}%
}
\def\mycopyrightnotice{%
    {\footnotesize  978-1-4799-6773-5/14/\$31.00 \textcopyright2017 Crown\hfill}
    \gdef\mycopyrightnotice{}
}

\usepackage{color}
\usepackage{enumerate,graphicx,epstopdf}
\usepackage{amsmath,amssymb,bm}
\usepackage{cite}
\usepackage{mathtools}
\usepackage{array}
\usepackage{algorithm,algorithmic}
\usepackage{color}
\usepackage{booktabs}

\newtheorem{ass}{Assumption}
\newtheorem{rmk}{Remark}

\newtheorem{dfn}{Definition}

\newtheorem{prp}{Proposition}
\newtheorem{cor}{Corollary}
\newtheorem{thm}{Theorem}

\newcommand{\reals}{\mathbb{R}}
\newcommand{\sqr}{$\hfill\square$}
\newcommand{\gph}{\mathrm{gph}~}

\newcommand\eps[0]{\varepsilon}



\makeatletter
\newcommand*\titleheader[1]{\gdef\@titleheader{#1}}
\AtBeginDocument{%
  \let\st@red@title\@title%
  \def\@title{%
    \bgroup\normalfont\large\centering\@titleheader\par\egroup
    \vskip1.5em\st@red@title}
}
\makeatother

\title{A Semismooth Predictor Corrector Method for Real-Time Constrained Parametric Optimization with Applications in Model Predictive Control
\thanks{The authors are with the University of Michigan, Ann Arbor. Email:\{dliaomcp,mnicotr,ilya\}@umich.edu.  This research is supported by the National Science Foundation Award Number  CMMI 1562209.}}
\titleheader{This is the accepted version of the paper:\\
D. Liao-McPherson, M. Nicotra, and I. Kolmanovsky, ``A Semismooth Predictor Corrector Method for Real-Time Constrained Parametric Optimization with Applications in Model Predictive Control'', 57th IEEE Conference on Decision and Control, Miami, Florida, USA, 2018, \textcopyright IEEE}
\author{Dominic Liao-McPherson, Marco M. Nicotra, Ilya V. Kolmanovsky}

\begin{document}
\maketitle

\begin{abstract}
Real-time optimization problems are ubiquitous in control and estimation, and are typically parameterized by incoming measurement data and/or operator commands. This paper proposes solving parameterized constrained nonlinear programs using a semismooth predictor-corrector (SSPC) method. Nonlinear complementarity functions are used to reformulate the first order necessary conditions of the optimization problem into a parameterized non-smooth root-finding problem. Starting from an approximate solution, a semismooth Euler-Newton algorithm is proposed for tracking the trajectory of the primal-dual solution as the parameter varies over time. Active set changes are naturally handled by the SSPC method, which only requires the solution of linear systems of equations. The paper establishes conditions under which the solution trajectories of the root-finding problem are well behaved and provides sufficient conditions for ensuring boundedness of the tracking error. Numerical case studies featuring the application of the SSPC method to nonlinear model predictive control are reported and demonstrate the advantages of the proposed method.
\end{abstract}

\section{Introduction}
Real-time optimization has the potential to improve the capabilities of many engineered systems. The associated real-time optimization problems can be treated in the framework of parameterized nonlinear programming (PNLP). An important example is the one arising from model predictive control (MPC), where the control action is generated by solving an optimal control problem (OCP) at each sampling instant \cite{rawlings2009model,grune2017nonlinear}. In this context, the OCP typically depends on time-varying parameters such as state measurements and/or operator commands. As a result, the solution of the PNLP needs to be computed as the parameters vary over time, generating a so-called ``solution trajectory''.

In MPC and real-time optimization subsequent problems are typically related. Hence, provided the OCP is appropriately designed, the similarities between OCPs at subsequent sampling instances can be exploited to significantly reduce the computational resources required to implement MPC. In the literature, these methods are often referred to as fast or suboptimal MPC methods, sensitivity methods, and continuation/homotopy methods. We will refer to all of these as ``solution tracking methods''.

Many of the concepts used to develop solution tracking algorithms are based on continuation approaches for smooth nonlinear equations \cite{allgower2012numerical}. An early continuation method specifically for MPC is C/GMRES \cite{ohtsuka2004continuation} which combines a continuation approach with a Krylov solver and leads to an efficient algorithm for unconstrained parameterized OCPs. 

In \cite{zavala2008fast,zavala2009advanced} the sensitivity theory for nonlinear programs was used to develop solution tracking algorithms for fast receding horizon estimation and the advanced step method for MPC. A related algorithm, based on the neighboring extremal theory of optimal control, is the IPA-SQP method\cite{ghaemi2009integrated}. These methods consider inequality constraints but assume that the optimal solution is continuously differentiable with respect to the parameter and, as a result, tend to encounter difficulties in the presence of active set changes.

In \cite{zavala2010real} the authors used the framework of parameterized generalized equations to develop solution tracking algorithms for inequality constrained problems without assuming continuous differentiability of the solution trajectory. The differentiability assumption was replaced by Robinson's strong regularity property \cite{robinson1980strongly}. The solution trajectories of generalized equations under pointwise strong regularity assumptions were studied in \cite{dontchev2013euler} and a sequential convex programming approach was proposed in \cite{dinh2012adjoint}. Sensitivity and predictor corrector methods were developed in \cite{jaschke2014fast} and \cite{kungurtsev2014sqp}, respectively, for tracking solution trajectories of PNLPs when the strong regularity assumption does not hold. Due to the weaker assumptions, both methods require the solution of additional linear programs to compute search directions. A solution tracking method for distributed problems was proposed in \cite{hours2016parametric} and \cite{wolf2016fast} provides a survey of the topic.

In this paper we propose a solution tracking algorithm based on nonsmooth calculus. The necessary conditions for optimality of a parameterized NLP are converted to a system of nonsmooth equations using nonlinear complementarity functions \cite{sun1999ncp}, resulting in a parameterized nonsmooth root-finding problem. We present an algorithm which tracks solution trajectories of this root-finding problem using a semismooth predictor-corrector (SSPC) method. We present sufficient conditions under which the solution trajectories of the root-finding problem are well behaved and establish tracking error estimates for the algorithm. 


The SSPC methods has several advantages compared to existing methods. SSPC makes the same strong regularity assumptions as methods based on generalized equations \cite{zavala2010real,hours2016parametric,dinh2012adjoint,dontchev2013euler}; the subproblems generated by these methods are themselves optimization problems. In contrast, the subproblems generated by the SSPC algorithm are linear systems of equations, similarly to smooth calculus based methods \cite{zavala2009advanced,ghaemi2009integrated}. However, unlike smooth calculus based methods, the nonsmoothness caused by active set changes is naturally handled using generalized derivatives. As a result, SSPC is applicable to the same class of problems as the generalized equation methods but has lower complexity subproblems.

We make extensive use of Clarke's generalized Jacobian \cite{clarke1990optimization}, the notion of semismoothness \cite{qi1993nonsmooth}, and the semismooth Newton's method \cite{qi1993nonsmooth}. Two key papers regarding the application of nonsmooth Newton's methods to optimization problems are \cite{fischer1992special} and \cite{qi1997semismooth}. A survey on the topic of nonsmooth and smoothing Newton's methods can be found in \cite{qi1999survey}.

The contents of this paper are as follows. In Section~\ref{ss:problem_formulation} we discuss the problem setting. In Section~\ref{ss:concepts_from_ns} we review some concepts from nonsmooth analysis used in the paper. In Section~\ref{ss:pred_corr_steps} we reformulate the KKT conditions as a nonsmooth root-finding problem and derive the predictor and corrector steps used in SSPC. In Section~\ref{ss:path-following_algo} we assemble the predictor and corrector steps into a solution tracking algorithm. In Section~\ref{ss:numerical_experiments} we illustrate the utility of SSPC on a numerical example and provide some comparisons with sequential quadratic programming (SQP) methods. Finally, Section~\ref{ss:conclusions} contains some concluding remarks.

\section{Problem formulation and background on parameterized nonlinear programming} \label{ss:problem_formulation}
In this paper we consider parameterized nonlinear programs of the form,
\begin{subequations}  \label{eq:NLP}
\begin{gather}
\underset{z}{\mathrm{min.}} \quad f(z,p),\\
\mathrm{s.t.} \quad g(z,p) = 0, \\
c(z,p) \leq 0,
\end{gather}
\end{subequations}
where $f: \reals^n \times \reals^l \to \reals$, $g:\reals^n \times \reals^l \to \reals^m$, and $c:\reals^n \times \reals^l \to \reals^q$ are $\mathcal{C}^2$ in $z$ and $\mathcal{C}^1$ in $p$. Given a finite sequence of parameter values $\{p_k\}_{k = 0}^{M}$ and a pair $(p_0,x^*_0)$ which approximately minimizes \eqref{eq:NLP}, the objective is to track a solution trajectory of \eqref{eq:NLP}, denoted by $(p_k,x^*_k(p_k))$, as $k\to M$. We make the following assumption regarding the values of the parameter:
\begin{ass} \label{ass:Pbounded}
All $p$ lie in some compact convex set $\mathcal{P} \subset \reals^l$.
\end{ass}
The KKT conditions for \eqref{eq:NLP} are,
\begin{subequations} \label{eq:KKT}
\begin{gather}
\nabla_z L(z,\lambda,v,p) = 0,\\
g(z,p) = 0,\\
c(z,p) \leq 0,~v \geq 0,\label{eq:cmp_start}\\
v^T c(z,p) = 0 \label{eq:cmp_end},
\end{gather}
\end{subequations}
where $L(z,\lambda,v,p) = f(z,p) + g(z,p)^T \lambda + c(z,p)^T v$ is the Lagrangian and $\lambda \in \reals^m$ and $v \in \reals^q$ are the dual variables. Any primal-dual tuple $x = (z,\lambda,v)$ which satisfies \eqref{eq:KKT} is called a KKT point. To ensure that the minimizers of \eqref{eq:NLP} are necessarily KKT points we will apply an appropriate constraint qualification. Recall that the linear independence constraint qualification (LICQ) is said to hold at a point $(\bar{z},\bar{p})$ if
\begin{equation}
  \text{rank}~\begin{bmatrix}
  \nabla_z g(\bar{z},\bar{p})\\
  \nabla_z c(\bar{z},\bar{p})_i
  \end{bmatrix}
   = m + |I_a(\bar{z},\bar{p})|,~ i \in I_a(\bar{z},\bar{p})
\end{equation}
where $I_a(z,p) = \{i\in 1~...~ q~|~ c_i(z,p) =0\}$ denotes the index set of active constraints. Further, if a KKT point $(\bar{x},\bar{p})$ satisfying the LICQ also satisfies
\begin{equation}
  u^T \nabla_z^2 L(\bar{x}) u> 0,~\forall u \in \mathcal{K}_+(\bar{z},\bar{v},\bar{p}) \setminus \{0\},
\end{equation}
where $\mathcal{K}_+(x,v,p) = \{u\in \reals^n~|~ \nabla_z g(\bar{z},\bar{p}) = 0,~ \nabla_zc_i(\bar{z},\bar{p}) u \leq 0, i \in I^+_a(z,v,p), \nabla_z f(\bar{z},\bar{p})^T u \leq 0\}$, and $I_a^+(z,v,p) = I_a(z,p) \cap \{i~|~ v_i > 0\}$ then $\bar{x}$ it is said to satisfy the strong second order sufficient conditions (SSOSC). As detailed in e.g., \cite{nocedal2006numerical}, any KKT point which satisfies the SSOSC and the LICQ is a strict local minimizer of \eqref{eq:NLP}.\\

Since we seek to track minimizers of \eqref{eq:NLP} as the parameter varies, it is desirable to ensure that there exists at least one path\footnote{There could be multiple paths since \eqref{eq:NLP} is not necessarily convex.} and that any existing paths are ``well behaved''. This can be guaranteed by imposing regularity conditions on the problem. To do so we define the solution mapping,
\begin{equation}
	S: p \to S(p) = \{x~|~ \text{\eqref{eq:KKT} is satisfied} \},
\end{equation}
which may be multivalued, and use the concept of strong regularity \cite{robinson1980strongly} in a form which echoes \cite{dontchev2013euler}:
\begin{dfn}
A set valued mapping $F: \reals^N \rightrightarrows \reals^M$ with $(\bar{x},\bar{y}) \in \mathrm{gph}~F$ is said to be strongly regular at $\bar{x}$ for $\bar{y}$ if there exists neighbourhoods $U$ of $\bar{x}$ and $V$ of $\bar{y}$ such that the restricted inverse mapping $\tilde{F}^{-1}: V \to F^{-1}(V)\cap U$ is single valued and Lipschitz continuous on its domain.
\end{dfn}

In the context of parameterized optimization, a strongly regular solution is one where a (local) primal-dual solution of the optimization problem, $x^*$, is locally a Lipschitz continuous function of the parameter, i.e., $x^* = x^*(p)$. The following theorem gives necessary and sufficient conditions for a solution to be strongly regular.
\begin{thm}
A primal-dual solution $x^*$ of \eqref{eq:NLP} is strongly regular if and only if $x^*$ satisfies the LICQ and the SSOSC.
\end{thm}
\begin{proof}
See e.g., \cite[Prop 1.27 and 1.28]{izmailov2014newton} or \cite[Theorem 2G.8]{dontchev2009implicit}.
\end{proof}
\begin{cor} \label{corr:lipschitz_param}
For each strongly regular solution $(\bar{p},\bar{x})$ of \eqref{eq:NLP} there exists a neighbourhood $T$ of $\bar{p}$ and a constant $L_p(\bar{p},\bar{x})$ such that $x^*(p)$ is a function satisfying $||x^*(p) - \bar{x}|| \leq L_p ||p - \bar{p}||,~ \forall p \in T$.
\end{cor}
\begin{proof}
See e.g., \cite[Theorem 2B.1]{dontchev2009implicit}.
\end{proof}
Our primary regularity assumption is stated below; it excludes phenomena like bifurcations or local minima becoming stationary points as the parameter varies.
\begin{ass} \label{ass:str_reg}
(Pointwise strong regularity) The LICQ and SSOSC hold at all KKT points in $\mathcal{P}$.
\end{ass}

Corollary~\ref{corr:lipschitz_param} is used to establish boundedness of the solution tracking error (Theorem~\ref{thrm:tracking_error}). For general constrained optimization problems, we cannot expect solution trajectories to satisfy stronger smoothness properties than local Lipschitz continuity. Indeed, constraint activation/deactivation typically destroys differentiability and non-convex problems may have multiple local minima. Strong regularity also imparts desirable properties to the solution mapping, in particular it establishes that $S(p)$ is comprised of finitely many isolated Lipshitz continuous trajectories \cite[Theorem 3.2]{dontchev2013euler}. 

\begin{rmk}
The LICQ and SSOSC are standard assumptions in the convergence theory of sequential quadratic programming type (SQP) algorithms \cite{nocedal2006numerical}, though convergence can be established under weaker conditions \cite{bonnans1994local}. Similarly, the pointwise strong regularity condition is a common assumption in literature on time varying optimization, e.g., \cite{zavala2010real,hours2016parametric,dinh2012adjoint,dontchev2013euler}. Lipschitz continuity of the primal variable and objective function can be established under weaker conditions \cite{kojima1980strongly} which is exploited in e.g., \cite{jaschke2014fast,kungurtsev2014sqp}. However, the resulting degeneracy of the dual variables complicates the algorithms, requiring the solution of quadratic and linear programs to determine search directions. In contrast, the use of generalized derivatives allows for methods that require only the solution of linear systems of equations.
\end{rmk}

\section{Some concepts from nonsmooth analysis} \label{ss:concepts_from_ns}
In this section we briefly review some concepts from non-smooth analysis that will be used later in the paper. Consider a function $G:\reals^n \to \reals^m$ which is locally Lipschitz continuous on an open set $U\subset \reals^n$. Rademacher's theorem \cite{rademacher1919partielle}
states that $D$, the set of points where $G$ is differentiable, is dense. Clarke's generalized Jacobian \cite{clarke1990optimization} is defined as follows
\begin{multline}
  \partial G(x) = \text{convh}~\{J\in \reals^{m\times n}|~ \exists \{x^k\} \subset D : \\ \{x^k\} \rightarrow x,~ \{\nabla G(x_k)\} \rightarrow J\},
\end{multline} 
where $\text{convh}~A$ denotes the convex hull of $A$. Note that the generalized Jacobian is a set of matrices, $\nabla G(x)\in\partial G(x)$, wherever $G$ is differentiable and, whenever $G$ is continuously differentiable, it reduces to $\partial G(x) = \{\nabla G(x)\}$. A key notion in the analysis of nonsmooth Newton's methods is semismoothness \cite{qi1993nonsmooth}. The function $G$ is said to be semismooth at $x$ if $G$ is Lipschitz in a neighbourhood of $x$, directionally differentiable in every direction and satisfies the following\footnote{We refer readers unfamiliar with big and little O notation to \cite[A.2]{nocedal2006numerical}, or \cite[A.2]{izmailov2014newton}.},
\begin{equation}
  \underset{J \in \partial G(x+\xi)}{\text{sup}} ||G(x+\xi) - G(x) - J\xi|| = o(||\xi||),
\end{equation}
if the right hand side is replaced by $O(||\xi||^2)$ then $G$ is said to be strongly semismooth at $x$.

\section{The predictor and corrector steps} \label{ss:pred_corr_steps}
The SSPC algorithm is based on mapping the KKT necessary conditions to a nonsmooth system of equations using what is known as a nonlinear complementarity (NCP) function \cite{sun1999ncp}. An NCP function $\psi : \reals^2 \to \reals$ has the property that
\begin{equation}
	\psi(a,b) = 0 \Leftrightarrow a \geq 0, b\geq 0, ab = 0,
\end{equation}
which can be used to convert complementarity systems into equations. A common example of an NCP function is the minimum function $\psi(a,b) = \mathrm{min}(a,b)$ implemented in \cite{qi1997semismooth}. 

Following \cite{qi1997semismooth}, we use an NCP function to map the KKT conditions \eqref{eq:KKT} to a system of nonsmooth equations. We define the mapping
\begin{equation} \label{eq:Fmapping}
	F(x,p) = \begin{bmatrix}
	\nabla_z L(z,\lambda,v,p)\\
	g(z,p)\\
	\psi(-c_1(z,p),v_1)\\
	\vdots\\
	\psi(-c_q(z,p),v_q)
	\end{bmatrix} = \begin{bmatrix}
		\nabla_z L(z,\lambda,v,p)\\
	g(z,p)\\
	\phi(-c(z,p),v)
	\end{bmatrix},
\end{equation}
where $x = (z,\lambda,v)$ is the primal-dual tuple and $\phi$ collects the last $q$ components of $F$. Due to the properties of the NCP function, the roots of $F$ coincide with the KKT points of \eqref{eq:NLP}. Solution trajectories $x(p) \in S(p)$ of \eqref{eq:NLP} can thus be constructed by tracking solutions of $F(x,p) = 0$ as the parameter varies. This mapping is semismooth \cite{qi1997semismooth}, thus this can be accomplished using a semismooth Euler-Newton predictor-corrector algorithm. Since $F$ is semismooth we can approximate it to first order in a neighbourhood of any $(\bar{x},\bar{p})$ as,
\begin{equation}
F(x,p) \approx F(\bar{x},\bar{p}) + V(p-\bar{p}) + B(x-\bar{x}),
\end{equation}
where $V\in \partial_p F(\bar{x},\bar{p})$ and $B\in \partial_xF(\bar{x},\bar{p})$ in a process analogous to Taylor expansion. From this approximation we can derive Euler predictor and Newton corrector steps by setting the approximation to zero. The resulting steps are,
\begin{subequations} \label{eq:pc}
\begin{align}
\text{Predictor:}& \quad F_{k-1} + V_{k-1} \Delta p_k + \hat{B}_{k-1} (x^-_k - x_{k-1}) = 0, \label{eq:predictor}\\
\text{Corrector:}& \quad F_{k}^- + \hat{E}_{k} (x_k - x_k^-) = 0, \label{eq:corrector}
\end{align}
\end{subequations}
the predictor solves \eqref{eq:predictor} for $x_k^-$ and the corrector solves \eqref{eq:corrector} for $x_k$. The matrices used in \eqref{eq:pc} are defined as follows: $V_{k-1} \in \partial_pF(x_{k-1},p_{k-1})$, $\hat{E}_k \in \partial_xF(x^-_k,p_k) + \Sigma_{k}$, $B_{k-1} \in \partial_xF(x_{k-1},p_{k-1}) + \Sigma_{k-1}$, $\Delta p_k = p_k - p_{k-1}$, $F_k^- = F(x_k^-,p_k)$, and $F_{k-1} = F(x_{k-1},p_{k-1})$.
These expressions include errors terms, $\Sigma_k, \Sigma_k{k-1}$, which represent e.g., regularization.

The generalized Jacobian of \eqref{eq:Fmapping} is given by all matrices of the form \cite[Prop 3.26]{izmailov2014newton}:
\begin{equation} \label{eq:dF_x}
	\partial_x F(x,p) = \begin{bmatrix}
		\nabla_z^2 L(x,p) & \nabla_z g(z,p)^T & \nabla_z c(z,p)^T\\
		\nabla_z g(z,p) & 0 & 0\\
		-C \nabla_z c(z,p) & 0 & D
	\end{bmatrix},
\end{equation}
where $C = \text{diag}(\gamma)$ is a diagonal matrix with elements satisfying
\begin{equation}
	\gamma_i \in \begin{cases}
	[0,1], & \text{if } v_i = -c_i(z,p),\\
	\{1\}, & \text{if } v_i > -c_i(z,p),\\
	\{0\}, & \text{if } v_i < -c_i(z,p),
	\end{cases}
\end{equation}
and $D = diag(1 - \gamma)$. All elements of $\partial_x F$ are guaranteed to be non-singular in the vicinity of a strongly regular solution (Proposition~\ref{prp:properties}). The Jacobian $\partial_p F$ consists of all matrices of the form
\begin{equation} \label{eq:dF_p}
	\partial_p F(x,p) = \begin{bmatrix}
		\nabla_{pz}L(x,p)\\
		\nabla_{p}g(z,p)\\
		-C \nabla_{p} c(z,p)
	\end{bmatrix},
\end{equation}
where $C$ is the same matrix as in \eqref{eq:dF_x}.

We add regularization to the algorithm in order to improve numerical conditioning by using $\hat{D} = D + \delta I$, for some $\delta \geq 0$, in place of $D$ in \eqref{eq:dF_x}. The regularization terms are extremely important in practice because elements of $\partial_x F$ can easily become ill-conditioned, causing the SSPC algorithm to diverge. We have observed that even a small amount of regularization, e.g., $\delta \approx 10^{-6}$ to $10^{-12}$, reliably handles this issue; likely because, near strongly regular solutions, all elements of $\partial_x F$ are guaranteed to be invertible in exact arithmetic.

Thanks to regularity assumptions made in Section~\ref{ss:problem_formulation}, it is possible to establish error bounds for the predictor and corrector steps which are summarized in the following theorem:
\begin{thm} \label{thrm:tracking_error}
Suppose that $x_{k-1}$ lies within a neighbourhood $\bar{X}_{k-1}$ of $x^*_{k-1} \in S(p_{k-1})$. Define the errors $e_k = x_k - x_k^*$ and $e_k^- = x^-_k - x^*_k$. Then there exists a neighbourhood $\bar{T}_{k-1}$ of $p_{k-1}$, $Z_k$ of $x_k^*$ and positive constants $\alpha,\beta,\sigma,$ and $\eta$ such that
\begin{subequations} \label{eq:error_bounds}
\begin{gather}
	||e_k^-|| \leq \alpha ||e_{k-1}||^2 + \beta ||e_{k-1}|| ||\Delta p_k|| + \sigma||\Delta p_k||^2,\\
	||e_k|| \leq \eta ||e_k^-||^2,
\end{gather}
\end{subequations}
provided $p_k \in \bar{T}_{k-1}$, $\bar{X}_{k-1}$ is sufficiently small, and $x_k^- \in Z_k$.
\end{thm}
\begin{proof}
See appendix.
\end{proof}

Theorem~\ref{thrm:tracking_error} demonstrates the existence of a region within which the tracking error is guaranteed to remain bounded. It generalizes the quadratic convergence estimates of the classical Newton's method to the setting of parameter dependent semismooth problems. These contraction estimates provide a theoretical  foundation for the SSPC algorithm. However, as is typical with Newton's method, the results are local and the proof provides no insight into how to estimate the sizes of the various neighbourhoods; these concerns are usually handled by adding safeguards. In the context of SSPC safeguarding the method requires limiting $\Delta p_k$, which may be difficult in many real-time applications. In the next section, we suggest a constructive method for overcoming this issue by interpolating between $p_{k-1}$ and $p_k$, similarly to \cite{jaschke2014fast}, and taking multiple steps along the resulting path.

\section{A path-following algorithm for real-time optimization} \label{ss:path-following_algo}
In this section we present a solution tracking algorithm for quickly computing solutions of PNLPs in real-time. We assume that a measurement $p_k$ becomes available at each sampling instance $k$, and that an approximate solution $x_{k-1}$ of the PNLP for the parameter value $p_{k-1}$ was computed at the previous timestep. The objective is then to compute $x_{k}$ satisfying $||F(x_k,p_k)||\leq \eps$ as quickly as possible.

Since the parameter change $\Delta p_k = p_k - p_{k-1}$ may be too large to ensure that the tracking bounds of Theorem~\ref{thrm:tracking_error} hold, we propose to generate a path connecting $p_k$ and $p_{k-1}$ and take smaller steps along the path. Similarly to \cite{jaschke2014fast} we construct a path depending on a homotopy parameter $t \in [0,1]$ as $P(t) = p_{k-1} + t\Delta p_k$. We assume that a constant $\kappa$ is known such that if $||\Delta p|| \leq \kappa$ then the conditions of Theorem~\ref{thrm:tracking_error} can be satisfied if $\eps$ is chosen correctly. The tolerance $\kappa$ can be thought of as the maximum allowable parameter variation. The SSPC algorithm then traverses the path between $p_{k-1}$ and $p_k$, alternating between a predictor step and corrector loop, using a uniform stepsize $h$ such that the inequality $||\Delta p_k||h\leq \kappa$ is satisfied. The SSPC algorithm is summarized in Algorithm~\ref{algo:SSPC}. 

\begin{algorithm}[H]
\caption{SSPC: Semismooth Predictor-Corrector}
\begin{algorithmic}[1] \label{algo:SSPC}
 \renewcommand{\algorithmicrequire}{\textbf{Input:}}
 \renewcommand{\algorithmicensure}{\textbf{Output:}}
 \REQUIRE $\delta_0$, $\eps$, $p_k$, $p_{k-1}$, $x_{k-1}$, $\kappa$
 \ENSURE  $x_k$
  \STATE $\Delta p_k = p_k - p_{k-1}$, $M \gets \texttt{max(1,ceil}(||\Delta p_k||/\kappa))$
  \STATE $h \gets 1/M$, $x \gets x_k$,  $\delta \gets \delta_0$
  \FOR{$i = 1~...~M$}
  \STATE $p^+ \gets P(t+h),~ p \gets P(t)$
  \STATE $\delta \gets \text{min}(\delta,||F(x,p)||)$ \label{step:reg_update}
  \STATE Compute $\hat{B} \in \partial_xF(x,p) + \Sigma(\delta)$, $V\in \partial_pF(x,p)$
  \STATE $x\gets x -\hat{B}^{-1}[h V\Delta p_k + F(x,p)]$,
  \WHILE{$||F(x,p^+)|| > \eps$}
    \STATE $\delta \gets \text{min}(\delta,||F(x,p^+)||)$
  	\STATE Compute $E \in \partial_x F(x,p^+) + \Sigma(\delta)$
    \STATE $x\gets x -E^{-1}F(x,p^+)$
  \ENDWHILE
  \STATE $t \gets t+h$
  \ENDFOR
  \RETURN $x$
 \end{algorithmic}
 \end{algorithm}

Note that the convergence properties of Algorithm~\ref{algo:SSPC} are identical to those of a semismooth Newton's method \cite{qi1993nonsmooth} and follow directly from the pointwise strong regularity assumption, the isolation of solution trajectories\cite[Theorem 3.2]{dontchev2013euler}, the convexity of $\mathcal{P}$, and Theorem~\ref{thrm:tracking_error}.

\begin{rmk}
The uniform grid algorithm serves to illustrate the concepts and performs well in our numerical studies but requires that $\kappa$ be treated as a tuning parameter. We have observed that SSPC is quite robust to the choice of $\kappa$. However, an algorithm with an adaptive step size, e.g., along the lines of \cite[Chapter 6]{allgower2012numerical}, is expected to be more robust and/or efficient than Algorithm~\ref{algo:SSPC} and is a topic of future work.
\end{rmk}

\section{Numerical Experiments} \label{ss:numerical_experiments}
\subsection{Spacecraft Attitude Control}
In this section we illustrate the performance of SSPC using a numerical example where we control the orientation of a rigid satellite using nonlinear MPC (NMPC). The attitude dynamics of a rigid spacecraft are given by the Euler equations,
\begin{equation}
	J\dot{\omega} + \omega^\times J \omega = u,
\end{equation}
where $\omega\in \reals^3$ is the vector of angular velocities expressed in a body fixed frame, $J = diag(918,920,1365)$, is the inertia matrix and $u \in \reals^3$ are external control moments \cite{de2012spacecraft}.
We choose a 3-2-1 Euler angle sequence as the orientation representation, the kinematic equations are:
\begin{equation}
	\dot{\theta} = S(\theta) w,~ S = \begin{bmatrix}
		1 & \sin(\theta_1) \tan(\theta_2) & \cos(\theta_1)\tan(\theta_2) \\
		0 & \cos(\theta_1) & -\sin(\theta_1)\\
		0 & \sin(\theta_1) \sec(\theta_2) & \cos(\theta_1) \sec(\theta_2)
	\end{bmatrix}.
\end{equation}
The equations of motion can then be written as
\begin{equation}
	\dot{\xi} = f_c(\xi,u) = \begin{bmatrix}
		J^{-1}(-\omega^\times J \omega + u)\\
		S(\theta) \omega
	\end{bmatrix},
\end{equation}
where $\xi = [\omega^T~\theta^T]^T$ is the state vector. We descritize the equations of motion using explicit Euler integration, i.e., $\xi_{k+1} = f_d(\xi_k,u_k) = \xi_k + \tau f_c(\xi_k,u_k)$, where $\tau = 3 [s]$ is the sampling period. The objective is to stabilize the satellite in a target orientation given by $r$, the reference vector. We consider the following optimal control problem,
\begin{subequations}
\begin{gather}
\underset{\xi,u,s}{min.}  \quad J(\xi,u,s) = ||\xi_N-r||_P^2  + \sum_{i=0}^{N-1} \ell(\xi_i,u_i,s_{i+1})\\
s.t \quad \xi_{i+1} = f_d(\xi_i,u_i), ~ i = 0, ~... ~,N-1,\\
c_{\xi}(\xi_i) \leq s_i, ~ i = 1, ~ ... ~ N,\\
c_{u}(\xi_i) \leq 0, ~ i = 0,~ ... ~ N-1,\\
-s_i \leq 0, ~ i = 1, ~ ... ~ N,
\end{gather}
\end{subequations}
where $N$ is the prediction horizon, $c_{\xi}(\xi) = [\xi^T-\xi^T_{ub} ~~\xi^T_{lb} - \xi^T]^T$, and $c_{u}(u) = [u^T-u^T_{ub} ~~u^T_{lb} - u^T]^T$, are the constraints and $\ell(\xi,u) = ||\xi-r||^2_Q + ||u||^2_R + \gamma s$ is the cost function\footnote{$Q = 10\text{diag}([10,10,10,1,1,1])$, $R = \text{diag}([0.1,0.1,0.1])$, $\gamma = 10$.}. The terminal penalty matrix, $P$, is chosen as the solution of the discrete time algebraic Riccati equation with the dynamics linearized about the origin. Linearly penalized slacks have been incorporated into the OCP to ensure feasibility. We consider two slew maneuvers, one with and one without active state constraints. For both cases the reference trajectory is given by
\begin{equation}
	r(t) = \begin{cases}
	[0 ~ 0 ~ 0 ~ 15^\circ ~ 30^\circ ~ -20^\circ]^T, & t < 120~s,\\
	[0~ 0 ~ 0 ~ 0 ~ 0 ~ 0]^T, & t > 120~s,
	\end{cases}
\end{equation}
and $\xi(0) = 0$. The constraints imposed during both cases are summarized in Table~\ref{tab:cstr}.  


\begin{table}[h]
\centering
\caption{The constraints used in the skew maneuver simulations.}
\label{tab:cstr}
\begin{tabular}{|c|c|c|}
\hline
 & Case 1 & Case 2\\ \hline
$\xi_{ub}$ & $360[1~1~1~1~1~1]^T$ & $[1.15~1.15~1.15~30~30~0]^T$\\\hline
$\xi_{lb}$ & $-360[1~1~1~1~1~1]^T$ & $-[1.15~1.15~1.15~0~0~20]^T$\\\hline
$u_{ub}$ & $[2~2~2]^T$ & $[2~2~2]^T$\\\hline
$u_{lb}$ & $-[2~2~2]^T$ & $-[2~2~2]^T$\\\hline
\end{tabular}
\end{table}

Closed-loop simulation results of Cases 1 and 2 with $N = 15$ can be found in Figures~\ref{fig:case1} and \ref{fig:case2}, respectively. A zoomed in view illustrating state constraint activation can be found in Figure~\ref{fig:cstr_cmp}. The NMPC controller successfully drives the spacecraft orientation to the desired setpoints while enforcing state and control constraints. Note that Case 2 was designed to be challenging numerically due to the presence of (i) multiple state constraint (de)activations, and (ii) infeasibility at certain steps (which is handled by the slack variables, but still causes ill-conditioning).

\begin{figure}[htbp]
	\centering
	\includegraphics[width=0.48\textwidth]{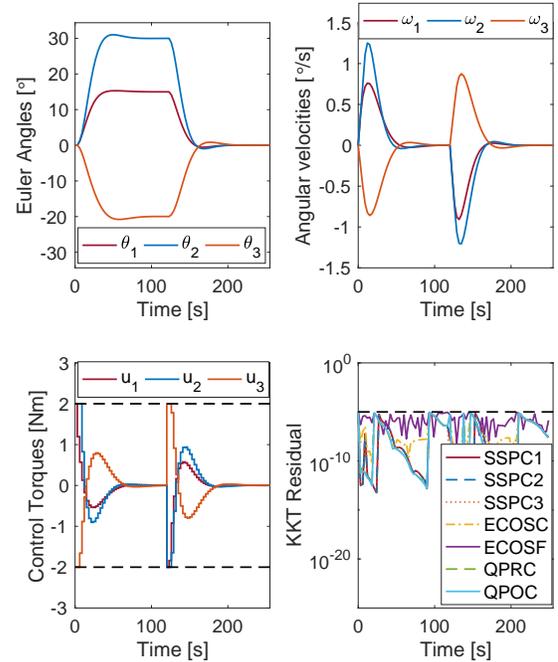}
	\caption{Case 1: The response of the closed-loop system to the slew-maneuver command.}
	\label{fig:case1}
\end{figure}

\begin{figure}[htbp]
	\centering
	\includegraphics[width=0.48\textwidth]{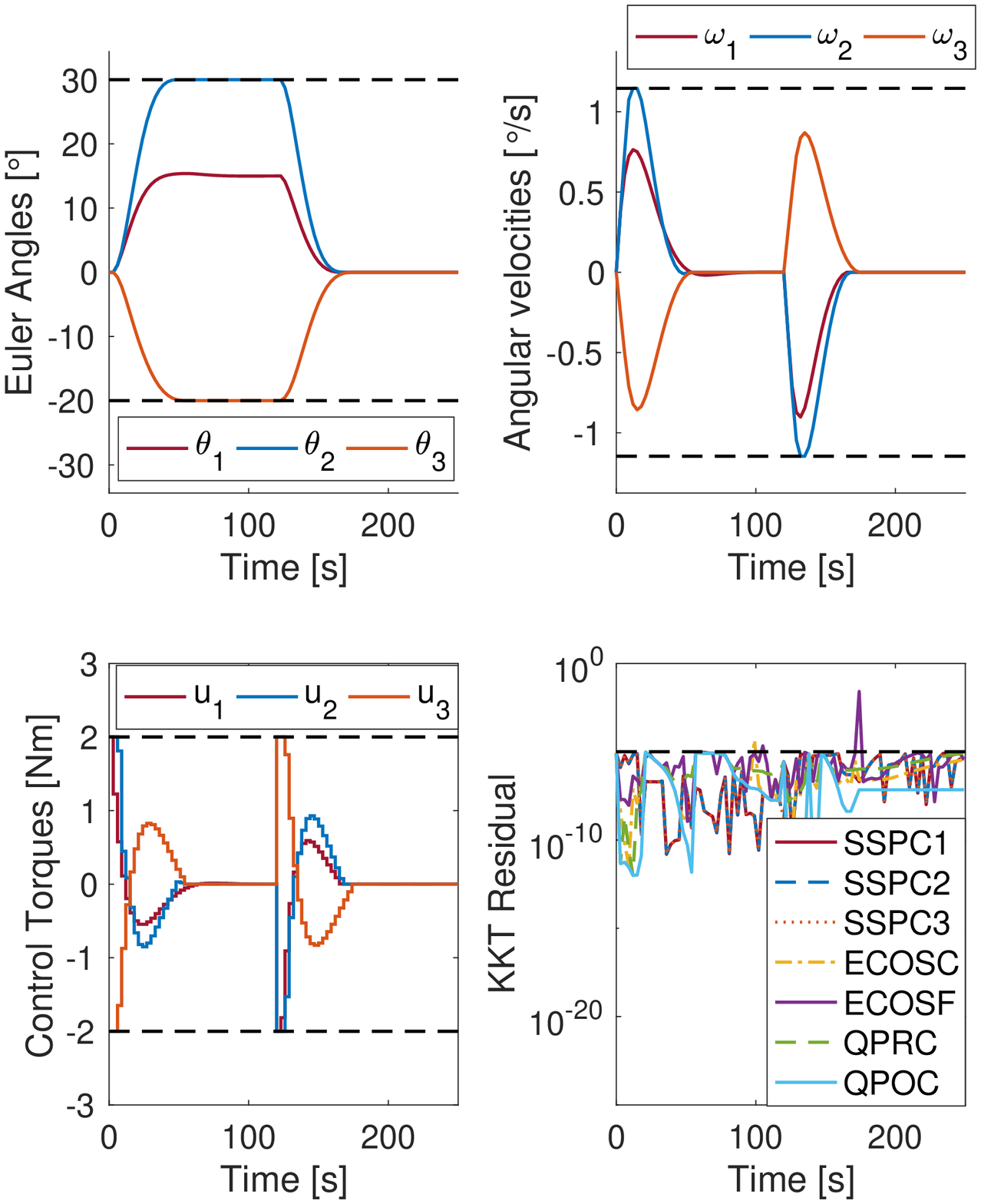}
	\caption{Case 2: The response of the closed-loop system to the slew-maneuver command.}
	\label{fig:case2}
\end{figure}

\begin{figure}[htbp]
	\centering
	\includegraphics[width=0.48\textwidth]{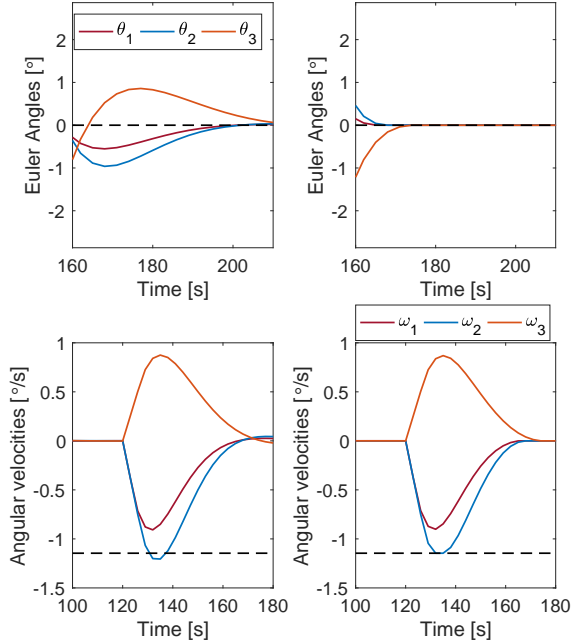}
	\caption{A comparison between the closed-loop responses of case 1 (left) and case 2 (right) illustrating, orientation constraint activation (top) and angular velocity constraint activation (bottom).}
	\label{fig:cstr_cmp}
\end{figure}

\subsection{Comparisons with Sequential Quadratic Programming}
To illustrate the utility of SSPC we perform numerical comparisons against the sequential quadratic programming (SQP) method. When warmstarted, the SQP method is equivalent to applying the Josephy-Newton method to the generalized equation reformulation of \eqref{eq:KKT} \cite{kungurtsev2014sqp,izmailov2014newton} and is thus a reasonable initial benchmark\footnote{Various state of the art QP solvers are also readily available ensuring that SSPC is benchmarked against efficient implementations.}. All necessary derivatives were computed automatically using CASADI \cite{Andersson2013b}. The SSPC method was implemented in native MATLAB code. The linear systems for the predictor and corrector were condensed, first by eliminating the $\xi$ variables using the equality constraints see e.g., \cite{andersson2013condensing}, then by using the Schur complement method with the matrix $D$ in \eqref{eq:dF_x} used as the pivot. The resulting condensed linear systems are sometimes referred to as the normal equations form of the originals, see e.g., \cite[Section 14.2]{nocedal2006numerical}. 

We implemented a standard SQP algorithm, \cite[Algorithm 4.13]{izmailov2014newton} using the augmented Lagrangian Hessian matrix $H_i = \nabla_z^2L(x_i,p_k) + \rho \nabla_zg(z_i,p_k)^T \nabla_z g(z_i,p_k)$ to maintain convexity of the QPs. A fixed penalty parameter $\rho = 1000$ was used throughout. We use three different state of the art QP solvers: i) ECOS, an interior point based SOCP solver specifically designed for embedded use \cite{domahidi2013ecos}, ii) qpOASES, an active set based strategy widely used for MPC \cite{Ferreau2014}, and iii) the MATLAB 2017b builtin \texttt{quadprog} using the \texttt{interiorpoint-convex} algorithm. At each timestep the SQP algorithm was initialized using the solution from the previous sampling instance.

We compare seven different configuration: (1) SSPC1, Algorithm~\ref{algo:SSPC} with $\kappa = 1$, (2) SSPC2, Algorithm~\ref{algo:SSPC} with $\kappa = 0.5$, (3) SSPC3, Algorithm~\ref{algo:SSPC} with $\kappa = 0.1$, (4) ECOSC, SQP with the QP solved by ECOS in condensed form\cite{andersson2013condensing}, (5) ECOSF, SQP with the QP solved by ECOS (6) QPRC, SQP with the QP solved by \texttt{quadprog} in condensed form\cite{andersson2013condensing}, (7) QPOC, SQP with the QP solved by qpOASES in condensed form\cite{andersson2013condensing}. All simulations were performed on a 2015 Macbook Pro with a 2.8GHz i7 processor and 16 GB of RAM running MATLAB 2017b. Execution times were measured using \texttt{tic} and \texttt{toc} and averaged over 10 executions to compensate for variance caused by the operating system. All solvers were stopped when $||F(x,p_k)|| \leq 10^{-5}$. 

\begin{rmk}
Note that the value of $\kappa$ used in SSPC1 was specifically chosen to yield $h = 1$, SSPC1 should therefore be interpreted as SSPC with the interpolation step removed.
\end{rmk}

Traces of the KKT residual $||F(x_k,p_k)||$ are shown in Figures~\ref{fig:case1} and \ref{fig:case2}. All solvers were able to keep the KKT residual within the specified tolerance except for ECOSF which has some minor difficulties during Case 2. Execution time histories for both cases are shown in Figure~\ref{fig:exe_comp}. Overall SSPC1 appeared to performed best followed by SSPC2, SSPC3, and ECOSF. The results of additional numerical trials are reported in Tables~\ref{tab:case1} and \ref{tab:case2}. SSPC outperformed all other methods on Case 1. In Case 2 SSPC1 and SSPC2 outperformed the other methods on average for $N = 10$ and $N=15$. At the longest horizon length considered, $N = 25$, ECOSF becomes the most effective method, demonstrating that ECOS scales more efficiently than SSPC, likely due to the sophisticated sparse linear algebra it employs\cite{domahidi2013ecos}. SSPC1, which uses $h = 1$, encountered numerical difficulties, caused by Jacobian ill-conditioning, for Case 2, $N = 25$. SSPC2 and SSPC3 did not encounter these difficulties, demonstrating the usefulness of the interpolation procedure for improving robustness while only mildly degrading performance.

Overall, SSPC performed well compared to several state of the art methods. Notably, despite being implemented in native MATLAB code, SSPC was competitive with SQP algorithms that use ECOS and qpOASES, both of which are implemented in \texttt{C/C++}. During our investigations, we found that both SQP and SSPC were perfectly reliable if only control constraints were considered. In this scenario the LICQ can be guaranteed to hold a-priori, provided the bounds are non-degenerate. We observed occasional robustness problems with SSPC when state constraints were allowed to activate. Specifically, if enough constraints activated at the same time then elements of $\partial_x F$ would become effectively singular despite regularization\footnote{In these cases SQP also encountered some difficulties but eventually recovered.}. We suspect this is due to the LICQ not holding, invalidating the strong regularity assumption used to guarantee invertability of the elements of $\partial_xF$. Future work will focus on relaxing the LICQ assumption and investigating techniques to constructively guarantee invertability of the iteration matrices in order to improve robustness.

\begin{figure}[htbp]
	\centering
	\includegraphics[width=0.48\textwidth]{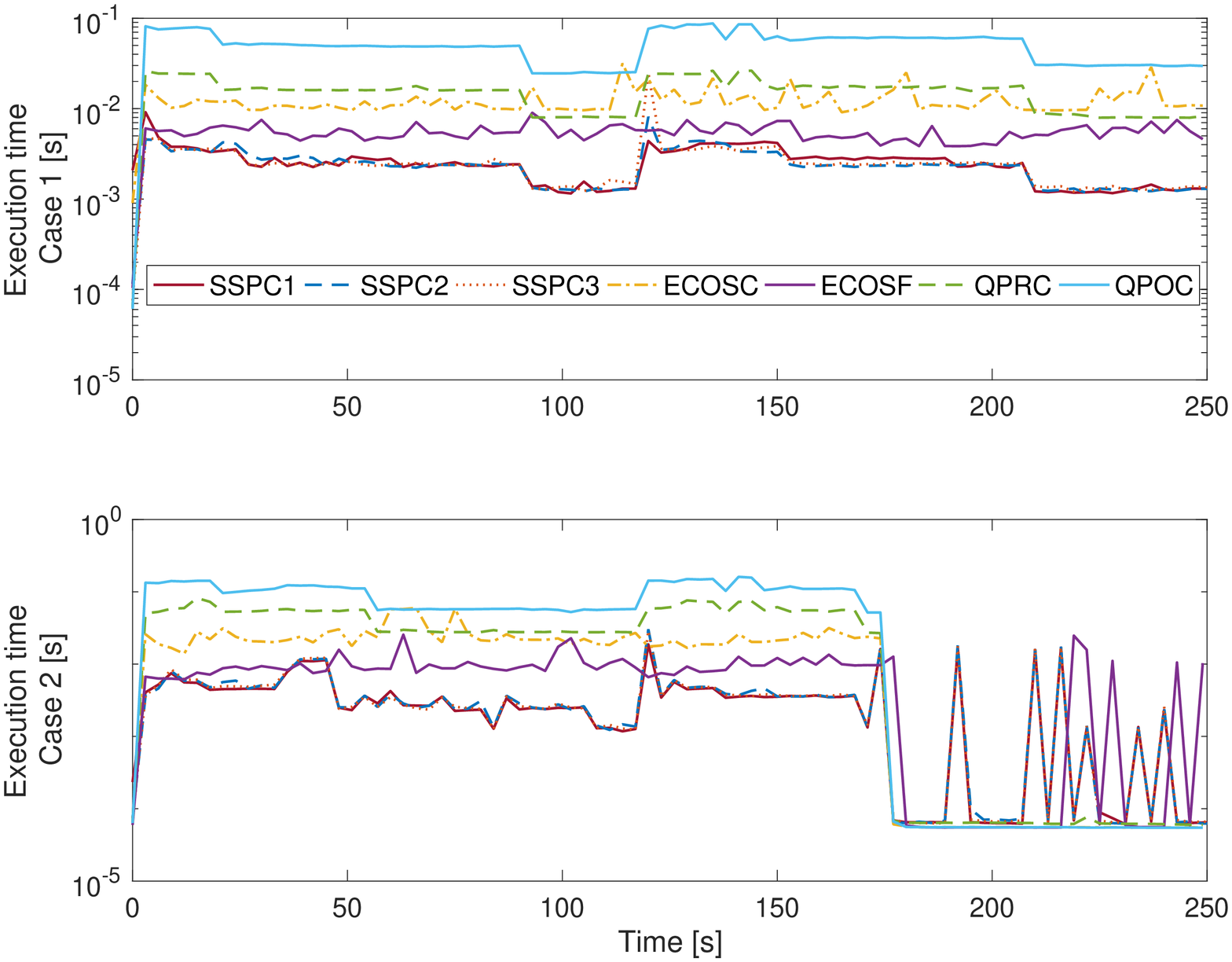}
	\caption{Execution time comparisons between the seven algorithms considered for case 1 (top) and case 2 (bottom).}
	\label{fig:exe_comp}
\end{figure}

\begin{table}
\centering
\caption{Numerical comparisons for case 1. All elements in a column are normalized by the first entry.}
\label{tab:case1}
\begin{tabular}{|c|c|c|c|c|c|c|}
\hline
 & \multicolumn{3}{|c|}{Max}&  \multicolumn{3}{|c|}{Average} \\ \hline
N & 10 & 15 & 25 & 10 & 15 & 25\\ \hline
Norm [ms] & 4.80 & 8.81 & 24.38 & 1.51 & 2.64 & 6.93 \\\hline
SSPC1& 1.35 & 0.67 & 0.74 & 1.04 & 0.96 & 1.07 \\\hline
SSPC2 & 1 & 1 & 1 & 1 & 1 & 1 \\\hline
SSPC3 & 2.89 & 3.38 & 2.90 & 1.03 & 1.21 & 1.13\\\hline
ECOSC & 3.36 & 4.67 & 2.56 & 4.26 & 4.97 & 4.58 \\\hline
ECOSF & 1.39 & 1.01 & 1.53 & 2.43 & 2.22 & 1.45 \\\hline
QPRC &  8.03 & 4.28 & 1.29 & 9.88 & 6.04 & 2.58 \\\hline
QPOC & 6.31 & 10.10 & 16.88 & 10.65 & 20.38 & 30.41 \\\hline
\end{tabular}
\end{table}

\begin{table}
\centering
\caption{Numerical comparisons for case 2. All elements in a column are normalized by the first entry.}
\label{tab:case2}
\begin{tabular}{|c|c|c|c|c|c|c|}
\hline
 & \multicolumn{3}{|c|}{Max}&  \multicolumn{3}{|c|}{Average} \\ \hline
N & 10 & 15 & 25 & 10 & 15 & 25\\ \hline
Norm [ms] & 9.75 & 33.35 & 67.29 & 1.78 & 4.22 & 15.85\\\hline
SSPC1& 0.94 & 0.66 & 5.77 & 1.00 & 0.97 & 15.58 \\\hline
SSPC3 & 1 & 1 & 1 & 1 & 1 & 1 \\\hline
SSPC5 & 2.0 & 0.93 & 1.39 & 1.05 & 0.98 & 1.02 \\\hline
ECOSC & 3.0 & 1.91 & 2.57 & 5.29 & 4.18 & 3.44 \\\hline
ECOSF & 1.25 & 0.91 & 0.71 & 2.30 & 1.88 & 0.87 \\\hline
QPRC & 4.18 & 2.47 & 3.82 & 8.80 & 7.92 & 6.54  \\\hline
QPOC &  5.71 & 4.97 & 10.28 & 12.78 & 16.67 & 18.11\\\hline
\end{tabular}
\end{table}



\section{Conclusions} \label{ss:conclusions}
In this paper we presented a semismooth predictor-corrector (SSPC) method for tracking solutions of parameterized constrained nonlinear programs. The method is simple, easy to code, has nice theoretical properties, and was shown to be competitive with an SQP algorithm which uses state of the art QP solver implementations.

Future work includes the following: Improving robustness of the method by relaxing the LICQ assumption, developing an efficient and robust adaptive algorithm for choosing the step sizes, investigating the use of SSPC for suboptimal MPC, and evaluation of the method on rapid prototyping hardware.


\appendix
\label{ss:convergence_proofs}
In this appendix we derive bounds on the tracking error of the predictor and corrector steps (Theorem~\ref{thrm:tracking_error}). The following proposition summarizes the properties of $F$ which will be used in the subsequent analysis.

\begin{prp} \label{prp:properties}
The mapping $F:\reals^{n+m+q} \times \mathcal{P} \to \reals^{n+m+q}$ has the following properties.
\begin{enumerate}
\item $F$ is locally Lipschitz continuous
\item $F$ is strongly semismooth
\item $F$ is CD regular \cite{qi1997semismooth} in the vicinity of any $(\bar{p},\bar{x})\in \gph{S}$, meaning that there exists a neighbourhood $X$ of $\bar{x}$ within which all $V \in \partial_x F(x,\bar{p})$ are nonsingular.
\end{enumerate}
\end{prp}
\begin{proof}
\textit{Result 1}: This follows from the continuous differentiability of all the functions in \eqref{eq:NLP} and the Lipschitz continuity of the \textrm{min} function.
\textit{Result 2}: \cite[Theorem 3.2]{qi1997semismooth}.
\textit{Result 3}: CD regularity is implied by strong regularity \cite[Theorem 4.2]{qi1997semismooth}. The remaining claims follow from the CD regularity of $F$ and \cite[Proposition 3.1]{qi1993nonsmooth}.
\end{proof}

\subsection{Proof of Theorem~\ref{thrm:tracking_error}} \label{ss:tracking_proof}
Consider a point $x_{k-1} \in X_{k-1}$, where $X_{k-1}$ is a neighbourhood of $x^*_{k-1} \in S(p_{k-1})$ within which all $B_{k-1}\in \partial_x F(x,p_{k-1})$ are nonsingular. This neighbourhood must exist by Proposition~\ref{prp:properties}. Now consider the predictor equation,
\begin{equation*}
	x_k^- = x_{k-1} - \hat{B}_{k-1}^{-1} [V_{k-1} \Delta p_k + F_{k-1}].
\end{equation*}
Performing some algebraic manipulations we obtain
\begin{gather*}
	-\hat{B}_{k-1} e_k^{-} = [\hat{B}_{k-1} (x_k^* - x_{k-1}) + V_{k-1} \Delta p_k + F_{k-1}]\\
	 =[B_{k-1} (x_k^* - x_{k-1}) + V_{k-1} \Delta p_k + F_{k-1} + \Sigma_{k-1} (x_k^* - x_{k-1})].
\end{gather*}
Due to the strong semismoothness of $F$ (Proposition~\ref{prp:properties}) we have that there exits a neighbourhood $Y_{k-1}$ of $(x_{k-1},p_{k-1})$ such that
\begin{equation} \label{eq:ss_predictor_lin}
	F(x,p) = F_{k-1} + \begin{bmatrix}
		B_{k-1} & V_{k-1}
	\end{bmatrix} \begin{bmatrix}
		x - x_{k-1}\\ p-p_{k-1}
	\end{bmatrix} + r,
\end{equation}
wherein the residual satisfies $||r|| \leq \gamma(||x - x_{k-1}||^2 + ||p - p_{k-1}||^2)$, $\forall (x,p) \in Y_{k-1}$. Applying \eqref{eq:ss_predictor_lin} with $x = x_k^*$ and $p = p_k$ yields
\begin{equation*}
	-e_k^- = B_{k-1}^{-1} [r + \Sigma_{k-1} (x_k^* - x_{k-1})],
\end{equation*}
where we have also used that $F(x_k^*,p_k) = 0$. Taking norms we obtain that, if $(x_k^*,p_k) \in Y_{k-1}$, then
\begin{multline}
	\frac{||e_k^-||}{||\hat{B}_{k-1}||} \leq \gamma ||x^*_k - x_{k-1}||^2 + \gamma ||\Delta p_k||^2 \\+ ||\Sigma_{k-1}||~||x^*_k - x_{k-1}||. \label{eq:Cproof}
\end{multline}
To proceed, we consider the term
\begin{subequations}
\begin{align}
	||x^*_k - x_{k-1}|| &\leq ||x_k^* - x_{k-1}^*|| + ||x_{k-1}^* - x_{k-1}||\\
	&\leq L_p ||\Delta p_k|| + ||e_{k-1}||, \label{eq:Vproof}
\end{align}
\end{subequations}
in \eqref{eq:Vproof} we have used Corollary~\ref{corr:lipschitz_param} to conclude that $x^*(p)$ is Lipschitz continuous on a set $T_{k-1}$, containing $p_{k-1}$, with constant $L_p$. Define the set $U_{k-1} = \bar{X}_{k-1} \times \bar{T}_{k-1}$ such that $U_{k-1} \subseteq Y_{k-1}$, $\bar{X}_{k-1} \subseteq X_{k-1}$, and $\bar{T}_{k-1} \subseteq T_{k-1}$. This is always possible due to \eqref{eq:Vproof}. Applying \eqref{eq:Vproof} to \eqref{eq:Cproof} we obtain that
\begin{multline*}
	\frac{||e_k^-||}{||\hat{B}_{k-1}||} \leq \gamma (L_p ||\Delta p_k|| + ||e_{k-1}||)^2 + \gamma ||\Delta p_k||^2 \\+ ||\Sigma_{k-1}||~(L_p ||\Delta p_k|| + ||e_{k-1}||).
\end{multline*}
The error induced by the Jacobian inexactness $\Sigma_{k-1}$ can be bounded due to the fact that, by construction (Step~\ref{step:reg_update} in Algorithm~\ref{algo:SSPC}), $||E_{k-1}|| \leq c||F_{k-1}||$ for some $c > 0$. Thus we have
\begin{align} \label{eq:E_bound}
	||E_{k-1}||\leq c||F_{k-1}|| \leq c L_F||e_{k-1}||,
\end{align}
for all $(x,p) \in U_{k-1}$ where we have used the Lipschitz continuity of $F$, with constant $L_F$, on the set $U_{k-1}$. Applying this result allows us to conclude that
\begin{multline*}
	\frac{||e_k^-||}{||\hat{B}^{-1}_{k-1}||} \leq \gamma (L_p ||\Delta p_k|| + ||e_{k-1}||)^2 + \gamma ||\Delta p_k||^2 \\+ c L_F||e_{k-1}|| (L_p ||\Delta p_k|| + ||e_{k-1}||),
\end{multline*}
provided $p_k \in \bar{T}_{k-1}$ and $x_{k-1} \in \bar{X}_{k-1}$.
Expanding and collecting terms we obtain that
\begin{equation*}
	||e_{k}^-|| \leq \alpha ||\Delta p_k||^2 + \beta ||e_{k-1}||~||\Delta p_k|| + \sigma ||e_{k-1}||^2,
\end{equation*}
where $\alpha = \gamma (1+L_p^2) ||\hat{B}_{k-1}^{-1}||$, $\beta = L_p (2+cL_F)||\hat{B}_{k-1}^{-1}||$, and $\sigma_ = (1+cL_F)||\hat{B}_{k-1}^{-1}||$.\\

\noindent Now consider the corrector equation
\begin{equation*}
	x_k = x_k^- - (\hat{E}_{k})^{-1} F_{k-1}^-,
\end{equation*}
performing some algebraic manipulation and exploiting the strong semimsoothness of $F$ we obtain that
\begin{gather*}
e_k = - (\hat{E}_{k})^{-1}[E_k(x^*_k - x^-) + F_k^- - \Sigma_ke_k^-],\\
e_k = - (\hat{E}_{k})^{-1}[r_k^- - \Sigma_ke_k^-],
\end{gather*}
where $||r_k^-|| \leq \zeta ||e_k^-||^2$ holds in a neighbourhood $Z_k$ of $x^*_k$. Taking norms and using \eqref{eq:E_bound} to bound the inexactness in the Jacobian we have that
\begin{equation*}
	||e_k|| \leq  \eta_k||e_k^-||^2, ~~\forall x_{k}^- \in Z_k
\end{equation*}
where $\eta = ||(\hat{E}_{k})^{-1}|| (\zeta+c L_F^-)$. \sqr

\bibliography{ss_pc}



\end{document}